\documentclass[12pt,a4paper,reqno]{amsart}
\usepackage{amsmath,amscd,amssymb,latexsym}
\usepackage{longtable}

\newtheorem{Thm}{Theorem}[section]
\newtheorem{Cor}[Thm]{Corollary}
\newtheorem{Lem}[Thm]{Lemma}
\newtheorem{Prop}[Thm]{Proposition}
\theoremstyle{definition}

\newtheorem{Exm}[Thm]{Example}

\newtheorem{Rk}[Thm]{Remark}

\begin{document}

\begin{center}
\bf{\LARGE{Splittability, non-splittability, and automatic splittability of metacyclic $p$-groups}}
\end{center}

\title[Metacyclic $p$-groups]{{}}
\author[A. Schweizer]{Andreas Schweizer}

\address{Andreas Schweizer,
Department of Mathematics Education, 
Kongju National University, 
Gongju, 33588 South Korea}
\email{schweizer@kongju.ac.kr}

\thanks{The author was supported by the Basic Science Research Program
through the National Research Foundation of Korea (NRF) funded by the 
Ministry of Education (No. 2022R1A2C1010487).}

\begin{abstract}
Let $P$ be a finite metacyclic $p$-group where $p$ is an odd prime. We refine the notions 
split metacyclic and non-split metacyclic group by distinguishing whether a metacyclic structure
(i.e. a cyclic normal subgroup $K$ of $P$ such that $P/K$ is also cyclic) is non-splittable or splittable
or automatically split. We show that these different types can be characterized in terms of $|K|$. We 
also determine which of these types can coexist on the same group $P$.
\end{abstract}

\maketitle 

\section{Introduction}

\noindent
If $G$ is a finite group, we write $|G|$ for its order, $exp(G)$ for its exponent,
$G'$ for its commutator subgroup, and $Z(G)$ for its centre. Also, $C_n$ denotes
the cyclic group of order $n$.
\par
A finite group $G$ is called metacyclic if it has a metacyclic structure, that is, 
a cyclic normal subgroup $K\trianglelefteq G$ such that $G/K$ is also cyclic. 
This implies of course that $G'\subseteq K$ and hence that $G'$ is also cyclic. 
Some older books, e.g. \cite{Zass} and \cite{Hall}, use the more restrictive definition 
that a metacyclic group is a group $G$ for which $G'$ and $G/G'$ are 
cyclic. But from Lemma \ref{cyclic} below one immediately sees that with that
definition the only metacyclic $p$-groups would be the cyclic ones.
\par
A metacyclic factorization $G=SK$ is a metacyclic structure $K\trianglelefteq G$
together with a cyclic subgroup $S$ of $G$ such that $K$ and $S$ together
generate $G$.
A metacyclic factorization $G=SK$ is called split if $S\cap K$ is trivial, i.e. 
$G=K\rtimes S$, otherwise non-split.
\par
A metacyclic group is called split metacyclic if it has at least one split metacyclic
factorization, otherwise non-split metacyclic. 
\par
Implicitly this already indicates that 
a metacyclic group might have different metacyclic structures, and a metacyclic 
structure might give rise to different metacyclic factorizations.
This can even happen when the underlying metacyclic group is a $p$-group,
which is the object of our interest. However, there are restrictions.

\begin{Thm} \label{YaLiu} {\rm \cite[Lemma 9]{YangLiu}} 
Let $P$ be a non-abelian metacyclic $p$-group of odd order. Suppose that both 
$P=SK$ and $P=S_1 K_1$ are split metacyclic factorizations. Then $|K|=|K_1|$.
Moreover, $SK_1$ and $S_1 K$ are also split metacyclic factorizations.
\end{Thm}

The first claim also follows from the more general result that for a split 
metacyclic $p$-group with $p$ odd all semidirect factorizations are 
isomorphic \cite[Theorem 1]{Kirt}.
\par
Building on \cite{YangLiu} we will generalize Theorem \ref{YaLiu}.
For that it is convenient to slightly generalize the definitions. 
\\ \\
{\bf Definitions:}
A metacyclic factorization $G=SK$ is called 
\begin{itemize}
\item[(a)]  {\bf splittable} if it is split or one can replace $S$ by a cyclic subgroup 
$R$ of $G$ such that $G=RK$ is a split metacyclic factorization;
\item[(b)]  {\bf non-splittable} if every metacyclic factorization $G=RK$ (with the 
given $K$) is non-split (including for $R=S$);
\item[(c)]  {\bf automatically split} if every metacyclic factorization $G=RK$ (with
the given $K$) is split (including for $R=S$).
\end{itemize}
Correspondingly, the underlying metacyclic structure $K\trianglelefteq G$ is also 
called splittable, resp. non-splittable, resp. automatically split.
\\ \\
The first surprise is that for $p$-groups the question whether a metacyclic structure 
$K\trianglelefteq P$ is splittable, non-splittable or automatically split can be immediately 
answered from the knowledge of $|K|$.

\begin{Thm} \label{main1}
Let $P$ be a non-abelian metacyclic $p$-group where $p$ is odd.
A metacyclic structure $K$ of $P$ is
\begin{itemize}
\item[(a)] non-splittable if and only if $|K|< exp(P)$ and $|P:K|< exp(P)$;
\item[(b)] splittable (but not automatically split) if and only if $|K| = exp(P) > |P:K|$;
\item[(c)] automatically split if and only if $|P:K|=exp(P)$.
\end{itemize}
\end{Thm}

\noindent
Secondly, there are restrictions which metacyclic structures can occur on the
same group.

\begin{Thm} \label{main2}
Let $p$ be an odd prime. Every non-abelian metacyclic group of order $p^N$ 
belongs to exactly one of the following $4$ types:
\begin{itemize}
\item[(a)]  {\bf automatically split:} Every metacyclic structure is automatically split. 
Such groups exist for every $N\geq 4$.
\item[(b)]  {\bf non-split metacyclic:} Every metacyclic structure is non-splittable.
Such groups exist for every $N\geq 6$. Their number is given by 
\cite[Theorem 2]{Liedahl}.
\item[(c)]  {\bf purely splittable:} Every metacyclic structure is splittable (but not 
automatically split). Such groups exist for $N=3$ and for every $N\geq 5$.
\item[(d)]  {\bf mixed type:} The group has at least one non-splittable and one 
splittable metacyclic structure.  The splittable structures are then necessarily 
not automatically split. Such groups exists for every $N\geq 4$.
\end{itemize}
\end{Thm}

\noindent The proof will show explicit examples for every $N$.
\par
Using Theorem \ref{exponent} below, from these two theorems we obtain

\begin{Cor} \label{same}
Let $P$ be a non-abelian metacyclic $p$-group with odd $p$. Then
\begin{itemize}
\item[(a)]  Every splittable metacyclic factorization $P=SK$ has the same $|K|$,
namely $|K|=exp(P)$ or $|K|=\frac{|P|}{exp(P)}$.
\item[(b)]  Every non-splittable metacyclic factorization $P=SK$ has the same
$|S|$, namely $|S|=exp(P)$.
\end{itemize}
\end{Cor}

\section{Metacyclic $p$-groups (including $p=2$)}\label{sect:allp}

\noindent
Although we are mainly interested in $p$-groups for odd primes $p$, the 
results in this section also hold for the case $p=2$.

\begin{Thm}\label{exponent}
If $P=SK$ is a metacyclic factorization of a $p$-group 
$P$, then 
$$exp(P)=\max\{|K|, |S|\}.$$
\end{Thm}

\begin{proof} 
This is shown in 
\cite[Lemma 7]{YangLiu} for odd $p$.
\par
The proof for $p=2$ is similar. Every element of $P$ can be written as
$ab$ with (not necessarily unique) $a\in S$, $b\in K$. Note that
$a^{-1}ba=b^r$ with an odd integer $r$. By induction over $t$ one easily
checks that $2^t$ divides $\frac{r^{2^t}-1}{r-1}$. So if $2^t =\max\{|K|, |S|\}$,
then 
$(ab)^{2^t} = a^{2^t} b^{1+r+r^2 +\cdots + r^{2^t -1}} = b^{\frac{r^{2^t}-1}{r-1}} =1$.
\end{proof}

The following is an immediate consequence.

\begin{Cor}\label{sqroot}
If $P$ is metacyclic, then $|P|$ divides $(exp(P))^2$.
\par 
If moreover $exp(P)=\sqrt{|P|}$, then every metacyclic factorization 
necessarily satisfies $|K|=|P:K|=exp(P)$ and is automatically split.
\end{Cor}

\begin{Lem}\label{cyclic}
Let $P$ be a finite $p$-group such that all abelian quotients of $P$ are 
cyclic. Then $P$ itself is cyclic.
\par 
In other words: A finite $p$-group is either cyclic or it has a quotient
$C_p \times C_p$. 
\par
Or even shorter: If $P/P'$ is cyclic, then $P$ is cyclic.
\end{Lem}

\begin{proof} 
\cite[Kapitel III, \S 7, Hilfsatz 7.1 c)]{Hupp}. 
\par
Alternatively, we offer a completely elementary proof by 
induction on $e$ where $|P|=p^e$. Let $x$ be an element of order $p$ in the
center of $P$. Every quotient of $P/\langle x\rangle$ also is a quotient of $P$.
So if $P/\langle x\rangle$ has a quotient $C_p \times C_p$ we are done. If not,
$P/\langle x\rangle$ is cyclic by induction. Fix an element $y\in P$ whose image
generates $P/\langle x\rangle$. If $y$ has order $p^e$, then $P$ is cyclic. 
Otherwise $y$ has order $p^{e-1}$ and we see 
$P \cong \langle x\rangle \times \langle y \rangle \cong C_p \times C_{p^{e-1}}$.
\par
Or even quicker: As a $p$-group $P$ has a nontrivial center $Z$. Every abelian 
quotient of $P/Z$ is an abelian quotient of $P$ and hence cyclic. Therefore by
induction $P/Z$ is cyclic. By a standard theorem this implies that $P$ is abelian,
and hence cyclic by the conditions of the theorem.
\end{proof}

The proof of the following theorem is largely taken from the literature. It is the 
first step towards our main results.

\begin{Thm}\label{cyclicsplitting}
Let $P$ be a finite metacyclic $p$-group.
\begin{itemize}
\item[(a)] If $P\cong K\rtimes S$ is a split metacyclic structure,
then $|K|=exp(P)$ or $|P:K|=exp(P)$.
\item[(b)] Conversely, let $K$ be a normal, cyclic subgroup of $P$. 
\begin{itemize} 
\item[(i)] If $|K|=exp(P)$, then $P/K$ is cyclic; and if moreover
$p$ is odd, the metacyclic structure $K\trianglelefteq P$ can be split.
\item[(ii)] If $|P:K|=exp(P)$ and $P/K$ is cyclic, then the metacyclic 
structure $K\trianglelefteq P$ splits automatically. 
\end{itemize} 
\end{itemize} 
\end{Thm}

\begin{proof}
(a) follows immediately from Theorem \ref{exponent}

(b)(i) The cyclicity of $P/K$ is shown in \cite[Lemma 8]{YangLiu} for odd 
$p$ and with a different proof in \cite[Lemma 1]{Berko} for all $p$. The 
splittability is proved in \cite[Theorem 1]{YangLiu} for odd $p$ and more 
generally in \cite[Theorem 2]{Berko} (with an extra condition that guarantees 
splitting if $p=2$). As pointed out in \cite[Example 1]{YangLiu}, the quaternion 
group is an instance of a $2$-group for which the splittability fails.

(b)(ii) Choose an element $a$ of $P$ whose image modulo $K$ generates $P/K$. 
Obviously the order of $a$ is at least $|P:K|$, on the other hand, it is at 
most $exp(P)$; so we have equality and $\langle a\rangle \cap K=\{1\}$. 
\end{proof}

\begin{Rk}\label{error} 
Part (b)(i) of Theorem \ref{cyclicsplitting} is also claimed in \cite[Theorem 2]{YangLiu}
to hold more generally for metacyclic groups of odd order. This can however not 
be true. Denote by $G_{21}$ the non-abelian group of order $21$ and consider 
$$G=C_{147} \times G_{21}.$$
Then $G$ is split metacyclic with kernel $C_3 \times C_7$ and 
$S=C_{49} \times C_3$. Also $exp(G)=147$ and $K=C_{147}$ (left direct factor) 
is a normal subgroup, whose quotient is however not cyclic.
\par 
The problem in the proof of \cite[Theorem 2]{YangLiu} seems to be the application 
of \cite[Lemma 9]{YangLiu} (Theorem \ref{YaLiu}) to $N_q$ when $N_q$ is abelian. 
In that case we cannot always exclude the possibility that $Y_q =S_q$. 
At least, that's exactly what is happening in our counter-example. So 
it might be that \cite[Theorem 2]{YangLiu} can be rescued by adding an extra 
condition, for example that $G$ has no abelian Sylow $p$-subgroups. 
\par 
In that case one should also correct a typo in \cite[Corollary 2]{YangLiu}. 
The condition should be that $|S|$ divides $|K|$, and not the other way round.
\end{Rk}

\section{Metacyclic $p$-groups ($p$ odd)}\label{sect:p}

\noindent
In this section we will prove Theorems \ref{main1} and \ref{main2}. and some
additional results. This will however require some preparation.

\begin{Thm}\label{Sim}
{\rm \cite[Lemma 3.4]{Sim}}
Let $P$ be a finite metacyclic $p$-group where $p$ is an odd prime.
Assume that $P$ has a metacyclic factorization $P=SK$ with $|K|>1$.  
Define $\alpha$, $\beta$, $\gamma$, $\delta$ by
$$p^\alpha =|S:S\cap K|,\ \ p^\beta =|K:S\cap K|,\ \ p^\gamma =|K|,\ \ p^\delta =|K:P'|.$$
Then 
$$\beta\leq \gamma\leq \beta +\delta\ \ \ \ and\ \ \ \ 
0<\delta\leq \gamma\leq \alpha +\delta,$$
and one can find suitable generators $x$ and $y$ with $\langle x\rangle=S$
and $\langle y\rangle=K$ such that $P$ has a presentation
$$P=\langle x,y\ |\ x^{p^\alpha}=y^{p^\beta},\ y^{p^\gamma}=1,\ y^x=y^{1+p^\delta}\rangle.$$
Moreover, 
$$|P|=p^{\alpha+\gamma},\ \ |S|=p^{\alpha+\gamma-\beta},\ \ 
|P'|=p^{\gamma-\delta},\ \ |S\cap K|=p^{\gamma-\beta}.$$
\end{Thm}

\begin{proof}
This is proved in \cite[Lemma 3.4]{Sim}. The condition 
$\gamma \leq \alpha +\delta$ is not mentioned there; it comes from
the fact that $x^{p^\alpha}$ lies in $K$ and hence acts as identity
on $y$. Also note that the cases where $P$ is abelian are represented 
by $\delta =\gamma$.
\par 
Finally we explain the reason for $\delta>0$. If $\delta=0$, then $K=P'$, 
so $P/P'$ is cyclic, and hence $P$ is cyclic by Lemma \ref{cyclic}. But 
then $P'$ is trivial and $\delta=0$ contradicts $|K|>1$.
\end{proof}

It is important to keep in mind that a metacyclic $p$-group can have different
presentations (which might have different advantages). We will make use of that
in several proofs. 
   So $\alpha$, $\beta$, $\gamma$, $\delta$ are not invariants of $P$.  
But some expressions in $\alpha$, $\beta$, $\ldots$ are invariants, as some
parts of the next lemma show.

\begin{Lem}\label{cardinalities}
Let $P$ be a finite metacyclic $p$-group ($p$ odd) with center $Z$ and
a presentation as in Theorem \ref{Sim}. Then
\begin{itemize}
\item[(a)] $exp(P)=\max\{p^{\gamma}, p^{\gamma+\alpha-\beta}\}$,
\item[(b)] $Z=(S\cap Z)(K\cap Z)$, (This actually holds for any metacyclic group.)
\item[(c)] $|S\cap Z|=p^{\alpha+\delta-\beta}$,
\item[(d)] $|K\cap Z|=p^{\delta}$,
\item[(e)] $|Z|=p^{\alpha-\gamma+2\delta}$,
\item[(f)] $P'\subseteq Z$ if and only if $\gamma\leq 2\delta$,
\item[(g)] $P'\subseteq S\cap K$ if and only if $\beta\leq \delta$,
\item[(h)] $S\cap K\subseteq P'$ if and only if $\delta\leq \beta$,
\item[(i)] $S$ is normal in $P$ if and only if $\beta\leq\delta$.
\end{itemize}
\end{Lem}

\begin{proof}
(a) Theorem \ref{exponent}.
\par
(b) Consider $ab$ with $a\in S$ and $b\in K$. If $ab\in Z$, then in particular $ab$ 
commutes with everything in $S$. Since $a$ already commutes with everything in $S$,
$b$ must also commute with everything in $S$. But $b$ already commutes with everything 
in $K$, so $b\in Z$. Analogously $a\in Z$, so $Z\subseteq (S\cap Z)(K\cap Z)$.
The reverse inclusion is trivial.
\par
(c) From the relation $y^x =y^{1+p^{\delta}}$ we have
$x^{-p^k}yx^{p^k} = y^{(1+p^k)^{p^k}}$. So $x^{p^k}$ commutes with $y$ if
and only if $k\geq \gamma -\delta$. Hence $S\cap Z$ is generated by 
$x^{p^{\gamma-\delta}}$. 
\par
(d) similar to (c) we see that $K\cap Z$ is generated by $y^{p^{\gamma -\delta}}$.
\par
(e) From (b) we have $|Z|=\frac{|S\cap Z|\cdot |K\cap Z|}{|S\cap K|}$.
\par
(f) $P'\subseteq Z$ if and only if $P'\subseteq K\cap Z$. Now $P'$ and $K\cap Z$ are
both contained in $K$, which is a cyclic $p$-group. So $P'\subseteq K\cap Z$ if and only 
if $|P'|\leq |K\cap Z|$.
\par
(g) and (h) are similar to (f).
\par
(i) $S$ is normal in $P$ if and only if $P'\subseteq S$. So this is essentially statement (g).
\par
\end{proof}

\noindent
For later use we mention the following application.

\begin{Thm}\label{HeKue} {\rm \cite[Lemma 2.1]{HetKue}}
Let $P$ be a finite metacyclic $p$-group where $p$ is odd. Then
$$|P'|^2|Z(P)|=|P|.$$
\end{Thm}

\begin{proof}
We are not sure whether \cite{HetKue} is the first time this fact was 
discovered. In any case, there seem to be different ways to prove it. 
Our proof, which simply consists in looking up the cardinalities from 
Theorem \ref{Sim} and Lemma \ref{cardinalities}, admittedly does not 
furnish a lot of insight.
\end{proof}
\bigskip

\noindent
{\bf Proof of Theorem \ref{main1}:}
Since the conditions on the right side disjointly cover all possibilities, it suffices to only prove
the conclusions from right to left.
\par
(a) If $|K|<exp(P)$ and $|P:K|<exp(P)$, then $K$ cannot be splittable by Theorem \ref{exponent}.
\par
(b) If $|K|=exp(P)$, then $K$ is splittable by Theorem \ref{cyclicsplitting}(b)(i).
So there exists $x\in P$ with $P=K\rtimes\langle x\rangle$. In particular,
$ord(x)=|P:K|=p^t$, say. Under the condition $p^t <|K|$ we will now also construct
a non-split metacyclic factorization $P=S_1 K$. Take $S_1 = \langle xy^{-1}\rangle$
where $K=\langle y\rangle$. Obviously $P=S_1 K$, and we still have to show $|S_1|>p^t$.
Note that metacyclic $p$-groups with odd $p$ are regular, for example by
\cite[Kapitel III, \S 10, Satz 10.2. c)]{Hupp}.
But by
\cite[Kapitel III, \S 10, Satz 10.6. a)]{Hupp}
for a regular $p$-group $(xy^{-1})^{p^t} =1$ would be equivalent to 
$y^{p^t} = x^{p^t}$, which does not hold.
\par
(c) If $|P:K|=exp(P)$, the metacyclic structure splits automatically by 
Theorem \ref{cyclicsplitting}(b)(ii).
{} \hfill $\Box$
\\ \\
Part (b) of the following result is stated in \cite{YangLiu} as a Corollary.
We present a detailled proof, as we feel that it is not obvious.

\begin{Thm}\label{nonsplit}
Let $p$ be an odd prime and $P$ a metacyclic $p$-group with a presentation
$$P=\langle x,y\ |\ x^{p^\alpha}=y^{p^\beta},\ y^{p^\gamma}=1,\ y^x=y^{1+p^\delta}\rangle$$
as in Theorem \ref{Sim}. Then
\begin{itemize}
\item[(a)] The metacyclic structure with kernel $K=\langle y\rangle$ is 
non-splittable if and only if, in addition to the conditions in Theorem \ref{Sim},
we have $\beta<\gamma$ and $\beta<\alpha$.
\item[(b)] \cite[Corollary 1]{YangLiu} Every metacyclic structure on $P$ is non-splittable 
if and only if in addition to the conditions in Theorem \ref{Sim} and the conditions 
from part (a) we moreover have $\delta<\beta$, in other words, if and only if
$$0<\delta<\beta<\gamma\leq \beta+\delta <\alpha+\delta.$$  
\end{itemize}
\end{Thm}

\begin{proof}
(a) Since $|S\cap K|=p^{\gamma-\beta}$, the metacyclic factorization $P=SK$ 
itself is split if and only if $\beta=\gamma$.
\par 
Next, under the condition $\beta<\gamma$ we show that $1\to K\to P$ is splittable
if and only if $\beta\geq\alpha$. 
If $\beta<\alpha$, then $|S|=p^{\alpha+\gamma-\beta}>p^{\gamma}$ and by 
Lemma \ref{cardinalities} (a) we have $exp(P)=p^{\alpha+\gamma-\beta}>p^{\gamma}$. 
Since $\beta<\gamma$ we also have $exp(P)>p^{\alpha}$. So $exp(P)$ is different from
$|K|$ and $|P:K|$, and hence $1\to K\to P$ cannot be split by Theorem \ref{cyclicsplitting} (a).
\par 
If $\beta\geq\alpha$, then $exp(P)=p^{\gamma}=|K|$ by Lemma \ref{cardinalities} (a);
so $K\trianglelefteq P$ can be split by Theorem \ref{cyclicsplitting} (b)(i).
\\ \\
(b) Let the conditions be as in part (a). 
\par 
If $\delta\geq \beta$, then $x^{-1}yx=y^{p^{\delta}}y\in\langle 
y^{p^{\beta}}\rangle y=\langle x^{p^{\alpha}}\rangle y$, 
whence $yxy^{-1}\in x\langle x^{p^{\alpha}}\rangle$, 
so $S=\langle x\rangle$ is normal in $P$. Since $|S|=exp(P)$ by (the proof of)
part (a), the metacyclic structure $S\trianglelefteq P$ can be split. 
This includes the possibility $S=P$ (if $\beta=0$).
\par 
From now on we assume $\delta<\beta$. Since $S\cap K$ and $P'$ are both 
subgroups of $K$, of orders $p^{\gamma-\beta}$ resp. $p^{\gamma-\delta}$,
we see that $S\cap K$ is contained in $P'$.
Now $P/P'$ inherits a metacyclic factorization 
$P/P' \cong \widetilde{S}\widetilde{K}$ with kernel $\widetilde{K}=K/P'$
and $\widetilde{S}=S/(S\cap K)$. Actually, $\widetilde{S}\cap\widetilde{K}$ 
is trivial, and since $P/P'$ is abelian the split metacyclic structure is
direct, so
$$P/P'\cong \widetilde{K} \times \widetilde{S}\cong 
C_{p^{\delta}} \times C_{p^{\alpha}}.$$
On the other hand, if $P$ has a split metacyclic factorization with kernel
$Y$, then $P'\subseteq Y$ and $|Y|=p^{\alpha+\gamma-\beta}$ or $p^{\beta}$
by Theorem \ref{cyclicsplitting} (a). This would imply 
$$P/P'\cong C_{p^{\alpha+\delta-\beta}} \times C_{p^{\beta}}$$
or 
$$P/P'\cong C_{p^{\beta+\delta-\gamma}} \times C_{p^{\alpha+\gamma-\beta}}.$$
Neither one is isomorphic to $C_{p^{\delta}} \times C_{p^{\alpha}}$.
\end{proof}

\begin{Rk}\label{KorS}
As a by-product the proof of Theorem \ref{nonsplit} shows that if
$P$ is a split metacyclic group and $P=SK$ is any (not necessarily split(able))
metacyclic factorization, then $K$ or $S$ must be the kernel of a splittable 
metacyclic structure.  
\end{Rk}

\begin{Exm}\label{example4}
Using the conditions in Theorem \ref{Sim} we can easily determine all possible 
presentations of non-abelian metacyclic groups of order $p^4$. Note that being
non-abelian imposes the extra condition $\delta< \gamma$. We end up with
\\ \par
$P_1:\ \ \alpha=2,\ \beta=2,\ \gamma=2,\ \delta=1,\ exp(P)=p^2, \ |K|=p^2,\ |S|=p^2$
\par
$P_2:\ \ \alpha=2,\ \beta=1,\ \gamma=2,\ \delta=1,\ exp(P)=p^3,\ |K|=p^2,\ |S|=p^3$
\par
$P_3:\ \ \alpha=1,\ \beta=3,\ \gamma=3,\ \delta=2,\ exp(P)=p^3, \ |K|=p^3, \ |S|=p$
\par
$P_4:\ \ \alpha=1,\ \beta=2,\ \gamma=3,\ \delta=2,\ exp(P)=p,^3 \ |K|=p^3, \ |S|=p^2$.
\\ \\
Checking the conditions of Theorem \ref{nonsplit} we see that the factorization
$P_2 =SK$ is non-splittable, but that $P_2$ must also have a splittable factorization.
Actually, at least two, because a splittable metacyclic structure  always comes in
at least one split and one non-split factorization.
As $P_2$ cannot be isomorphic to $P_1$ because of the exponent, these splittable 
factorizations must be the ones given by $P_3$ and $P_4$.
\par
By the way, an isomorphism from $P_2$ to $P_3$ can be made explicit.
In the presentation of $P_2$ we take $K_1 =\langle x \rangle$, which is of order 
$p^3$, and $S_1 =\langle x^p y^{-1} \rangle$. Using the regularity of metacyclic 
$p$-groups as in the proof of Theorem \ref{main1}(b), we see that $x^p y^{-1}$
has order $p$. From $x^{-1}yx=y^{1+p}=x^{p^2}y$ we obtain 
$x^{x^p y^{-1}}=yxy^{-1}=x^{1+p^2}$.
\par
The isomorphism between $P_2$ and $P_4$ simply is switching $x$ and $y$.
\par
Summarizing, in accordance with the number in \cite{Liedahl} there are $2$ different 
non-abelian metacyclic groups of order $p^4$, namely $P_1$, which is automatically 
split, and $P_2\cong P_3\cong P_4$, which is of mixed type.  
\end{Exm}

\begin{Thm}\label{autosplit}
Let $P$ be a non-abelian, metacyclic $p$-group of odd order. Assume 
that $P$ has a normal cyclic subgroup $N$ such that $P/N$ is cyclic 
and $|P:N|=exp(P)$. Then every metacyclic structure $K\trianglelefteq P$ 
satisfies $|K|=|N|$ and is automatically split.
\end{Thm}

\begin{proof}
Theorem \ref{cyclicsplitting} (b)(ii) tells us that $N$ is the kernel of a metacyclic 
structure that splits automatically. Now let $K\trianglelefteq P$ be another 
metacyclic structure. 
\par 
If it can be split, we first split it, then apply Theorem \ref{YaLiu}, 
and then Theorem \ref{cyclicsplitting} (b)(ii) again. 
\par 
If it is non-splittable, by Lemma \ref{KorS} there also exists a split 
metacyclic structure whose kernel has order $exp(P)$. In this case 
Theorem \ref{YaLiu} implies $exp(P)=\sqrt{|P|}$ and Corollary \ref{sqroot} 
implies the rest.
\end{proof}

\begin{Exm}\label{example5}
Explicit calculations using Theorem \ref{nonsplit} show that there are $4$ non-isomorphic
non-abelian metacyclic groups of order $p^5$, and they are all split metacyclic.
This agrees with the numbers in \cite{Liedahl}. Concentrating on the split presentations 
we obtain
\\ \par
$P_5:\ \ \alpha=3,\ \gamma=2,\ \delta=1,\ |P'|=p, \ |Z|=p^3,\ exp(P)=p^3$
\par
$P_6:\ \ \alpha=2,\ \gamma=3,\ \delta=1,\ |P'|=p^2, \ |Z|=p,\ exp(P)=p^3$
\par
$P_7:\ \ \alpha=2,\ \gamma=3,\ \delta=2,\ |P'|=p, \ |Z|=p^3,\ exp(P)=p^3$
\par
$P_8:\ \ \alpha=1,\ \gamma=4,\ \delta=3,\ |P'|=p, \ |Z|=p^3,\ exp(P)=p^4$.
\\ \\
Note that $\beta=\gamma$ as $S\cap K$ is trivial.
\par
To avoid misunderstandings: We did not list the non-split factorizations of the 
splittable metacyclic structures.
\par
The $P_i$ with different $\gamma$ 
cannot be isomorphic by Theorem \ref{YaLiu}, and from $|Z|$ or $|P'|$ we see that
$P_6 \not\cong P_7$.
\par
Now let's parametrize the groups of order $p^5$ which have a non-splittable structure, 
simply by finding all parameters such that $\alpha+\gamma=5$ and $\beta$ satisfies the 
conditions from Theorem \ref{nonsplit} (a). We obtain
\par
$\alpha=3,\ \gamma=2,\ \delta=1,\ \beta=1,\ |P'|=p, \ |Z|=p^3,\ exp(P)=p^4,\ |K|=p^2,\ |S|=p^4$
\\ and
\par
$\alpha=2,\ \gamma=3,\ \delta=2,\ \beta=1, \ |P'|=p, \ |Z|=p^3,\ exp(P)=p^4, \ |K|=p^3, \ |S|=p^4$.
\\ \\
Either one must be isomorphic to one of the $4$ groups above, and because of the exponent
this can only be $P_8$. Consequently $P_8$ is of mixed type and $P_6$ and $P_7$ are purely 
splittable, whereas $P_5$ is of course automatically split.
\par
Also note that the non-splittable structures with $|K|=p^2$ and $|K|=p^3$ on the same group
$P_8$ show that Theorem \ref{YaLiu} does not hold for non-splittable metacyclic structures.
\end{Exm}

Before starting the proof of Theorem \ref{main2} we need an easy criterion to exclude the 
mixed typ case.

\begin{Prop}\label{PropP'Z}
Let $p$ be an odd prime and $P$ a finite split metacyclic $p$-group. If $P$ also has 
a non-splittable metacyclic structure, then 
$$P'\subseteq Z(P).$$
\end{Prop}

\begin{proof}
Let $K\trianglelefteq P$ be a non-splittable metacyclic structure and choose generators as 
in Theorem \ref{Sim}. Then by Theorem \ref{Sim} and Theorem \ref{nonsplit} (a) all
inequalities from Theorem \ref{nonsplit} (b) except $\delta <\beta$ are satified.
Since $P$ also has a split metacyclic structure, we must have $\delta\geq \beta$.
By the other inequalities from Theorem \ref{nonsplit} (b) this implies 
$\gamma\leq\beta +\delta\leq 2\delta$, which by Lemma \ref{cardinalities} (f) is
equivalent to $P'\subseteq Z(P)$. 
\end{proof}

\begin{Cor}\label{CorP'Z}
Let $p$ be an odd prime and $P$ a finite metacyclic $p$-group.
\begin{itemize}
\item[(a)] If $|P'|^3 >|P|$, then either every metacyclic structure on $P$ is
splittable or none is splittable.
\item[(b)] If $|Z(P)|^3 <|P|$, then either every metacyclic structure on $P$ is
splittable or none is splittable.
\end{itemize}
\end{Cor}

\begin{proof}
By Theorem \ref{HeKue} either condition, (a) or (b), implies $|P'|>|Z(P)|$,
so the condition $P'\subseteq Z(P)$ from Proposition \ref{PropP'Z} cannot
hold.
\end{proof}

The group $P_7$ from Example \ref{example5} shows that $P'\subseteq Z$ does not imply that 
$P$ has both, splittable and non-splittable metacyclic structures. In other words, the converse of 
Proposition \ref{PropP'Z} does not hold. And indeed, for one of the cases in our main proof we 
need an additional criterion.

\begin{Lem}\label{p^6} 
A metacyclic group $P$ of order $p^6$ with $exp(P)=p^4$ and $|Z(P)|=p^2$ 
cannot be of mixed type.
\end{Lem}

\begin{proof}
Assume that there is a non-splittable metacyclic structure.
From $|K|< exp(P)$ and $|P:K|< exp(P)$ we obtain $\gamma=3$, and hence $\alpha=3$.
Furthermore, Lemma \ref{cardinalities} gives $\beta=2$ (from $exp(P)=p^4$) and 
$\delta=1$ (from $|Z|=p^2$).
\par
This is indeed a valid non-splittable presentation of a group $P_9$, and 
Theorem \ref{nonsplit} or Proposition \ref{PropP'Z} says that $P_9$ cannot have 
a splittable metacyclic structure. 
\par
The other possibility is of course that our assumption was incorrect and the group has
no non-splittable structure at all. 
For example 
\par
$P_{10}:\ \ \alpha=2, \ \beta=4, \ \gamma=4, \ \delta=2$ 
\\
is such a (purely splittable) group.
\par
This means that the conditions in the lemma do not determine the group up to isomorphism.
\end{proof}

We indeed need the argument with the presentation of $P_9$ to conclude that $P_{10}$
is purely splittable. Working with the presentation of $P_{10}$ alone does not suffice, because
it does not know anything about non-splittable structures. By contrast, the content of
Theorem \ref{nonsplit} is that a non-splittable presentation knows whether there are
splittable ones.
\par
By the way, the group $P_9$ from the proof of Lemma \ref{p^6} shows that 
Proposition \ref{PropP'Z} is stronger than Corollary \ref{CorP'Z}.
\\ \\
{\bf Proof of Theorem \ref{main2}:}
The non-abelian metacyclic group of order $p^3$ is purely splittable as the conditions 
$\beta<\gamma$ and $\beta<\alpha$ for being non-split from Theorem \ref{nonsplit} 
cannot both hold.
\par
Theorem \ref{autosplit} shows that the automatically split metacyclic structures do not mix 
with the other ones. So the $4$ types listed in Theorem \ref{main2} already cover all groups.
\par
(a) As examples for automatically split metacyclic groups of order $p^N$ for every $N\geq 4$
we can take
\\
$\alpha =N-2,\ \beta=2,\ \gamma=2,\ \delta=1$ with $exp(P)=p^{N-2}=|Z|$.
\par
(b) Examples for non-split type for every $N\geq 6$ are
\\
$\alpha=N-3,\ \beta=2,\ \gamma=3,\ \delta=1$ with $exp(P)=p^{N-2}$ and $|Z|=p^{N-4}$,
\\
as can be checked with Theorem \ref{nonsplit}.
\par
(c) If $N$ is even, say $N=2k-2$, then 
\\
$\alpha=k-2,\ \beta=k,\ \gamma=k,\ \delta=2$ with $exp(P)=p^k$ and $|Z|=p^2$
\\
is a split presentation of a purely splittable $P$ with $|P|=p^N$ for every $N\geq 6$.
For $N>6$ this follows from Corollary \ref{CorP'Z}. For $N=6$ this is the group $P_{10}$
from the proof of Lemma \ref{p^6}. For $N=4$ there is no purely splittable group by 
Example \ref{example4}.
\par
For odd $N$, say $N=2k-1$, we take
\\
$\alpha=k-1,\ \beta=k,\ \gamma=k,\ \delta=1$ with $exp(P)=p^k$, $|Z|=p$
\\
and $|P|=p^N$ for every odd $N\geq 5$ and again use Corollary \ref{CorP'Z}.
The group of order $3$ has already been mentioned at the beginning of the proof.
\par
(d) For $N\geq 4$ the group
\\
$\alpha=N-2,\ \beta=1,\ \gamma=2,\ \delta=1$ with $exp(P)=p^{N-1}$ and $|Z|=p^{N-2}$
\\
is of mixed type by Theorem \ref{nonsplit}.
{} \hfill $\Box$
\\


\bigskip

\end{document}